\documentclass[11pt]{article}
\usepackage[width=15cm, height=20cm]{geometry}
\usepackage[colorlinks, allcolors=blue]{hyperref}

\usepackage{mathtools}
\newcommand{\eqdef}{\coloneqq}

\usepackage{amsmath,amssymb,amsfonts}
\newcommand{\bC}{\mathbb{C}}
\newcommand{\bR}{\mathbb{R}}
\newcommand{\bZ}{\mathbb{Z}}

\newcommand{\al}{\alpha}
\newcommand{\la}{\lambda}
\newcommand{\tht}{\vartheta}
\renewcommand{\Re}{\operatorname{Re}}

\newcommand{\arth}{\operatorname{arth}}

\newcommand{\arcth}{\operatorname{arccth}}
\newcommand{\arctanh}{\operatorname{arctanh}}
\newcommand{\medstrut}{\vphantom{\int_0^1}}
\newcommand{\bigstrut}{\vphantom{\int_{0_0}^{1^1}}}
\newcommand{\hstrut}{\mbox{}\ \mbox{}}
\newcommand{\laasympt}{\lambda^{\operatorname{asympt}}}

\newcommand{\mydoi}[1]{\href{https://doi.org/#1}{doi:#1}.}

\usepackage{amsthm}
\newtheorem{theorem}{Theorem}
\newtheorem{proposition}{Proposition}
\newtheorem{lemma}{Lemma}
\theoremstyle{remark}
\newtheorem{remark}{Remark}

\usepackage{url}

\usepackage{graphicx}

\usepackage{comment}

\begin{document}

\title{Eigenvalues of  tridiagonal Hermitian Toeplitz matrices\\ with perturbations in the off-diagonal corners}

\author{Sergei M. Grudsky,
Egor A. Maximenko,
Alejandro Soto-Gonz\'{a}lez}

\maketitle

\begin{abstract}
In this paper we study the eigenvalues
of Hermitian Toeplitz matrices 
with the entries $2,-1,0,\ldots,0,-\alpha$ in the first column.
Notice that the generating symbol depends on the order $n$ of the matrix.
If $|\alpha|\le 1$, then the eigenvalues belong to $[0,4]$ and are asymptotically distributed as the function $g(x)=4\sin^2(x/2)$ on $[0,\pi]$.
The situation changes drastically
when $|\alpha|>1$ and $n$ tends to infinity.
Then the two extreme eigenvalues (the minimal and the maximal one) lay out of $[0,4]$ and converge rapidly to certain limits determined by the value of $\alpha$, whilst all others belong to $[0,4]$
and are asymptotically distributed as $g$.
In all cases, we transform the characteristic equation to a form convenient to solve by numerical methods, and derive asymptotic formulas for the eigenvalues.



\end{abstract}

\section{Introduction}

Toeplitz matrices appear naturally in the study
of shift-invariant models with zero boundary conditions.
The general theory of such matrices is explained
in the books and reviews~\cite{BFGM2015,BG2005,BS1999,DIK2013,G2005,GS1958}.
Efficient formulas for the determinants
of banded symmetric Toeplitz matrices were found in \cite{Trench1987symm,Elouafi2014}.
The determinants, minors, cofactors, and eigenvectors of banded Toeplitz matrices
were recently expressed in terms
of skew Schur polynomials, see~\cite{A2012,MM2017}.
The individual behavior of the eigenvalues
of Hermitean Toeplitz matrices was investigated in~\cite{BBGM2018,BGM2017,BBGM2015,BBGM2017,DIK2012}.

Determinants of non-singular Toeplitz matrices
with low-rank perturbations were studied in~\cite{BFGM2015}.
The eigenvalues and eigenvectors of tridiagonal Toeplitz matrices
with some special perturbations
on the diagonal corners are computed in~\cite[Section~1.1]{BYR2006} and~\cite{NR2019}.
The determinants and inverses of a family
of non-symmetric tridiagonal Toeplitz matrices
with perturbed corners are computed
in~\cite{WJJS2019}.

Yueh and Cheng~\cite{YuehCheng2008} considered the tridiagonal Toeplitz matrices with four perturbed corners. Using the techniques of finite differences they derived the characteristic equation in a trigonometric form and formulas for the eigenvectors, in terms of the eigenvalues.
For some special values of the parameters, they computed explicitly the eigenvalues.
Unlike the present paper, \cite{YuehCheng2008} deals with arbitrary complex coefficients
and does not contain the analysis of the localization of the eigenvalues
nor approximate formulas for the eigenvalues.

In this paper we study the spectral behavior of one family of Hermitean tridiagonal Toeplitz matrices
with perturbations at the entries $(n,1)$ and $(1,n)$,
where $n$ is the order of the matrix.
For every $\al$ in $\bC$
and every natural number $n\ge 3$
we denote by $A_{\al,n}$
the $n\times n$ Toeplitz matrix
generated by the following function
which depends on the parameters $\al$ and $n$:
\begin{equation}\label{eq:genalpha}
-\overline{\al} t^{-n+1}-t^{-1}+2-t-\al t^{n-1}.
\end{equation}
After the change of variable $t=\exp(ix)$, the generating symbol results in
\[
h_{\al,n}(x)
=4\sin^2(x/2)-2\Re(\al\,e^{i(n-1)x}).
\]
For example, if $n=6$,
\[
A_{\al,6}=
\begin{bmatrix*}[r]
2 & -1 & 0 & 0 & 0 & -\overline{\al} \\
-1 & 2 & -1 & 0 & 0 & 0 \\
0 & -1 & 2 & -1 & 0 & 0 \\
0 & 0 & -1 & 2 & -1 & 0 \\
0 & 0 & 0 & -1 & 2 & -1 \\
-\al & 0 & 0 & 0 & -1 & 2
\end{bmatrix*}.
\]
Such matrices may appear in the study of one-dimensional
shift-invariant models on a finite interval,
with some special interactions
between the extremes of the interval.

For $\al=0$, the matrix $A_{\al,n}$
is the well studied tridiagonal Toeplitz matrix with the symbol
$g\eqdef h_{0,n}$, i.e.
\begin{equation}\label{eq:generating_function} 
g(x)\eqdef 4\sin^2(x/2).
\end{equation}
The characteristic polynomial of $A_{0,n}$
is $\det(\la I_n-A_{0,n})=U_n((\la-2)/2)$,
where $U_n$ is the $n$th Chebyshev polynomial
of the second type,
and the eigenvalues of $A_{0,n}$
are $g(j\pi/(n+1))$, $1\le j\le n$.
For general $\al$,
the characteristic polynomial of $A_{\al,n}$
can be expressed in terms of $U_n$ and $U_{n-2}$.
We are able to compute the eigenvalues of $A_{\al,n}$ explicitly only for $|\al|=1$ or $\al=0$.
For $|\al|\ne 1$, applying an appropriate trigonometric or hyperbolic change of variable,
we describe the localization of the eigenvalues,
transform the characteristic equation to a form that can be solved by the fixed point iteration, and get asymptotic formulas.

It turns out that the cases $|\al|<1$ (``weak perturbations'')
and $|\al|>1$ (``strong perturbations'') are essentially different:
if $|\al|>1$, then the extreme eigenvalues go outside the interval $[0,4]$ and need a special treatment.
Below we present the corresponding results separately, starting with the simpler case $|\al|<1$.
In Sections~\ref{sec:proofs_weak}
and~\ref{sec:proofs_strong} we give the corresponding proofs.
The case $|\al|=1$ can be viewed
as a limit of the case $|\al|<1$,
but the results in this case are much simpler,
see Section~\ref{sec:abs_alpha_eq_one}.
Finally, in Section~\ref{sec:numerical}
we discuss some numerical tests.

\section{Main results}
\label{sec:main_results}

The matrices $A_{\al,n}$ are Hermitean,
their eigenvalues are real,
and we enumerate them in the ascending order:
\[
\la_{\al,n,1}
\leq\la_{\al,n,2}
\leq\cdots
\leq\la_{\al,n,n}.
\]

\subsection{Main results for weak perturbations}

Recall that the function $g$
is defined by~\eqref{eq:generating_function}.
It strictly increases on $[0,\pi]$
taking values from $0$ to $4$.

\begin{theorem}[localization of the eigenvalues for weak perturbations]
\label{thm:weak_localization}
Let $\al\in\bC$, $|\al|\le 1$, $\al\notin\{-1,1\}$, and $n\ge 3$.
Then the matrix $A_{\al,n}$ has $n$ different eigenvalues belonging to $(0,4)$.
More precisely, for every $j$ in $\{1,\ldots,n\}$,
\begin{equation}\label{eq:weak_localization}
g\left(\frac{(j-1)\pi}{n}\right)
<\la_{\al,n,j}
<g\left(\frac{j\pi}{n}\right).
\end{equation}
\end{theorem}

The minimal value of the
generating function $h_{\al,n}$,
for $|\al|\le 1$,
may be strictly negative.
So, even the inequality $\la_{\al,n,1}>0$,
which is a very small part of Theorem~\ref{thm:weak_localization},
is not obvious.

Theorem~\ref{thm:weak_localization} implies
that the eigenvalues of $A_{\al,n}$,
as $n$ tends to $\infty$,
are asymptotically distributed as the values of the function $g$.
This follows also from
the theory of locally Toeplitz sequences
\cite{Tyrtyshnikov1996,Tilli1998,GS2017}
or, more specifically,
from Cauchy interlacing theorem,
since the matrices $A_{\al,n}$ are obtained from the tridiagonal Toeplitz matrices $A_{0,n}$ by low-rank perturbations.

Our next goal is to transform
the characteristic equation into a convenient form.
For every $\al$ in $\bC$ with $\al\notin\{-1,1\}$ and every integer $j$
we define the function $\eta_{\al,j}\colon[0,\pi]\to\bR$ by
\begin{equation}\label{eq:eta}
\eta_{\al,j}(x)
\eqdef -2\arctan\left( \left((-1)^{j+1}k_\al\cot{(x)} + \sqrt{k_\al^2\cot^2{(x)}+l_\al^2}\right)^{(-1)^j} \right),
\end{equation}
where
\begin{equation}\label{eq:kl_def}
k_\al
\eqdef
\frac{1-|\al|^2}{|1+\al|^2},\qquad
l_\al \eqdef
\frac{|1-\al|}{|1+\al|}.
\end{equation}
In fact, $\eta_{\al,j}$ depends only on $\al$ and on the parity of $j$.
Thus, for every $\al$
there are only two different functions:
$\eta_{\al,1}$ and $\eta_{\al,2}$.
These functions take values in $[-\pi,0]$.
See a couple of examples on Figure~\ref{fig:eta_weak_and_strong}.

\begin{figure}[ht]
\centering
\includegraphics{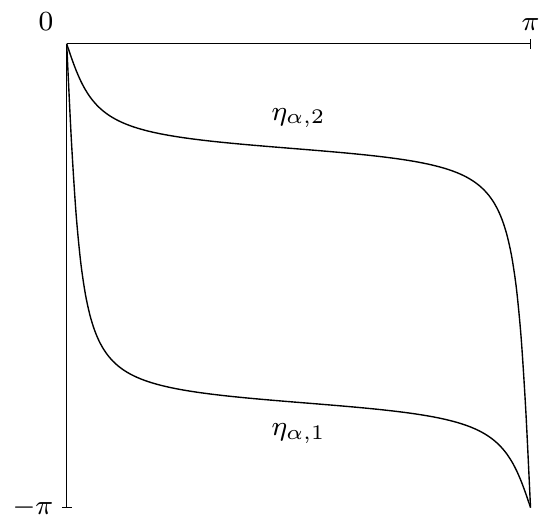}
\qquad
\includegraphics{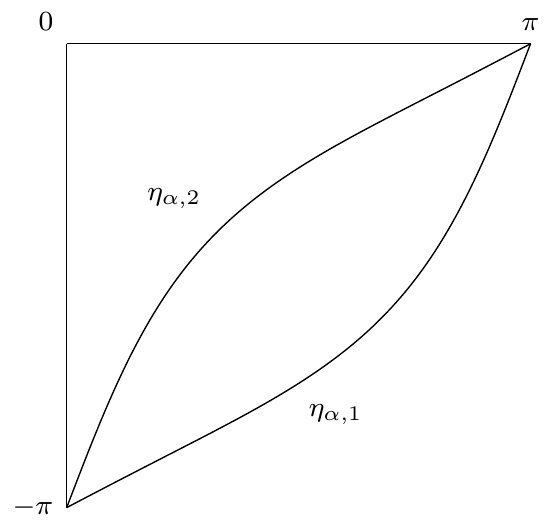}
\caption{Functions~\eqref{eq:eta} 
for $\al=0.7+0.6i$ (left)
and $\al=2+i$ (right).
\label{fig:eta_weak_and_strong}
}
\end{figure}

Motivated by~\eqref{eq:weak_localization},
we use the function $g$ as a change of variable
in the characteristic equation and put
$\tht_{\al,n,j}\eqdef g^{-1}(\la_{\al,n,j})$.
Inequality~\eqref{eq:weak_localization} is equivalent to
\begin{equation}\label{eq:weak_localization_theta}
\frac{(j-1)\pi}{n}
<\tht_{\al,n,j}<
\frac{j\pi}{n}.
\end{equation}

\begin{theorem}[characteristic equation for weak perturbations]
\label{thm:weak_equation}
Let $\al\in\bC$, $|\al|\le 1$, $\al\notin\{-1,1\}$, $n\ge 3$,
and $1\le j\le n$.
Then the number $\tht_{\al,n,j}$ satisfies
\begin{equation}\label{eq:weak_equation}
\tht_{\al,n,j}
=\frac{j\pi+\eta_{\al,j}(\tht_{\al,n,j})}{n}.
\end{equation}
\end{theorem}

In Section~\ref{sec:proofs_weak} we
show that for $n>N_1(\al)$, where
\begin{equation}\label{eq:N1}
N_1(\al)\eqdef \frac{4(|\al|+1)}{||\al|-1|},
\end{equation}
and for $1\le j\le n$,
the function in the right-hand side of~\eqref{eq:weak_equation} is contractive.
Hence, equation~(\ref{eq:weak_equation})
can be solved by the fixed point iteration.
Furthermore,
we use~\eqref{eq:weak_equation}
to derive asymptotic formulas for the eigenvalues.

For $|\al|\le 1$ and $1\le j\le n$,
define $\laasympt_{\al,n,j}$ by
\begin{equation}\label{eq:laasympt}
\begin{aligned}
\laasympt_{\al,n,j}
&\eqdef g\left(\frac{j\pi}{n}\right)
+\frac{g'\left(\frac{j\pi}{n}\right)
\eta_{\al,j}\left(\frac{j\pi}{n}\right)}{n}
\\
&\quad+\frac{g'\left(\frac{j\pi}{n}\right)
\eta_{\al,j}\left(\frac{j\pi}{n}\right)\eta_{\al,j}'\left(\frac{j\pi}{n}\right)
+\frac{1}{2}g''\left(\frac{j\pi}{n}\right)
\eta_{\al,j}\left(\frac{j\pi}{n}\right)^2}{n^2}.
\end{aligned}
\end{equation}

\begin{theorem}[asymptotic expansion of the eigenvalues for weak perturbations]
\label{thm:weak_asympt}
Let $\al\in\bC$, $|\al|\le 1$.
Then there exists $C_1(\al)>0$ such that
for $n$ large enough and $1\le j\le n$,
\begin{equation}\label{eq:weak_lambda_asympt}
|\la_{\al,n,j}-\laasympt_{\al,n,j}|
\le\frac{C_1(\al)}{n^3}.
\end{equation}
\end{theorem}

In other words, Theorem~\ref{thm:weak_asympt} claims that
$\la_{\al,n,j}=\laasympt_{\al,n,j}+O_\al\left(\frac{1}{n^3}\right)$, where the constant in the upper estimate of the residue term depends on $\al$.
For simplicity, we state and justify only this asymptotic formula with three exact terms,
but there are similar formulas with more terms.

Since the eigenvectors were found in~\cite{YuehCheng2008} for a more general matrix family,
we give the corresponding formulas in Propositions~\ref{prop:weak_eigenvectors} and~\ref{prop:strong_eigenvectors} without proofs.

\subsection{Main results for strong perturbations}

For the sake of simplicity,
we decided to state the following theorems only for large values of $n$.
As a sufficient condition,
we require $n>N_2(\al)$, where
\begin{equation}\label{eq:N2}
N_2(\al)\eqdef
\frac{20\log(|\al|+1)-4\log(\log(|\al|))}{\log|\al|}.
\end{equation}

\begin{theorem}[localization of the eigenvalues for strong perturbations]
\label{thm:strong_localization}
Let $\al\in\bC$, $|\al|>1$, and $n\ge N_2(\al)$.
Then
\[
\la_{\al,n,1}<0,\qquad
\la_{\al,n,n}>4,
\]
whereas for $2\le j\le n-1$
the eigenvalues $\la_{\al,n,j}$
belong to $(0,4)$ and satisfy~\eqref{eq:weak_localization}.
\end{theorem}

For some $\al$ with $|\al|>1$
and for small values of $n$, the eigenvalues
$\la_{\al,n,1}$ and $\la_{\al,n,n}$ can belong to $[0,4]$.
So, the condition $n>N_2(\al)$
in Theorem~\ref{thm:strong_localization} cannot be omitted.

In order to solve the characteristic equation for $\la<0$ and $\la>4$, we use the following changes of variables, respectively:
\begin{equation}\label{eq:g_plus_minus}
g_-(x)\eqdef -4\sinh^2(x/2),\qquad
g_+(x)\eqdef 4+4\sinh^2(x/2)\qquad(x>0).
\end{equation}

\begin{theorem}[characteristic equations for strong perturbations]
\label{thm:strong_equation}
Let $\al\in\bC$, $|\al|>1$, and $n\ge N_2(\al)$.
Then
\[
\la_{\al,n,1}=g_-(\tht_{\al,n,1}),\qquad
\la_{\al,n,n}=g_+(\tht_{\al,n,n}),
\]
where $\tht_{\al,n,1}$ is the unique positive solution of the equation
\begin{equation}\label{eq:eq_first}
x
=\arctanh\frac{2(|\al|^2-1)\tanh\frac{nx}{2}}%
{|\al+1|^2\tanh^2\frac{nx}{2}
+|\al-1|^2},
\end{equation}
and $\tht_{\al,n,n}$ is the unique positive solution of the equation
\begin{equation}\label{eq:eq_last}
x=\arctanh
\frac{2(|\al|^2-1)\tanh\frac{nx}{2}}%
{|\al+(-1)^n|^2\tanh^2\frac{nx}{2}
+|\al-(-1)^n|^2}.
\end{equation}
For $2\le j\le n-1$,
the eigenvalues $\la_{\al,n,j}$ can be found as in Theorem~\ref{thm:weak_equation}.
\end{theorem}

Moreover, in Section~\ref{sec:proofs_strong}
we prove that the right-hand sides of~\eqref{eq:eq_first} and~\eqref{eq:eq_last}
are contractive functions on the segment
\begin{equation}\label{eq:segment}
S_\al\eqdef
\left[\frac{\log|\al|}{2},\frac{3\log|\al|}{2}\right].
\end{equation}

Figure~\ref{fig:strong} shows an example.
In this figure,
we glued together three changes of variables
($g_-$, $g$, and $g_+$)
into one spline, using appropriate shifts or reflections.

\begin{figure}[ht]
\centering
\includegraphics{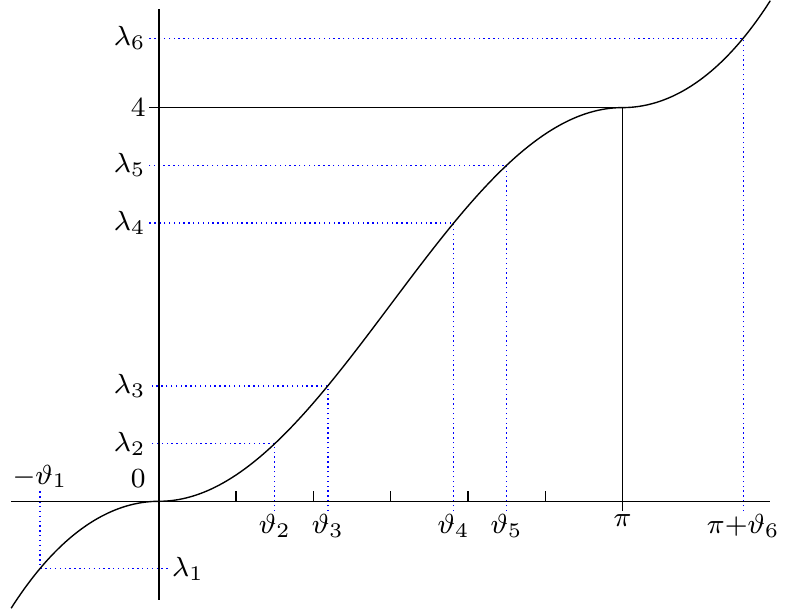}
\caption{The points $\tht_{\al,n,j}$
and the corresponding values of $\la_{\al,n,j}$ for $\al=2+i$, $n=6$.
Different scales are used for the axis.
\label{fig:strong}
}
\end{figure}

There is another way to see the changes of variables~\eqref{eq:g_plus_minus}:
after extending $g$ to an entire function,
$g_-(x)=g(ix)$ and $g_+(x)=g(\pi+ix)$.
So, the eigenvalues are obtained
by evaluating the function $g$
at some points belonging to
a piecewise linear path on the complex plane,
see Figure~\ref{fig:changes_of_variables}.

\begin{figure}[ht]
\centering
\includegraphics{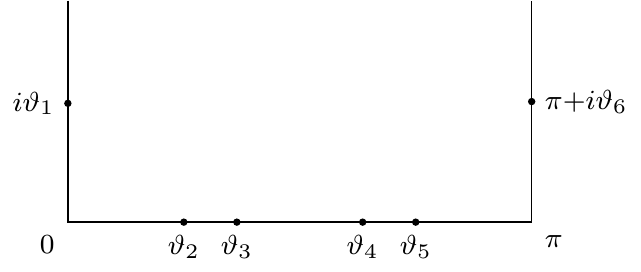}
\caption{Points on the complex plane
that yield the eigenvalues $\la_{\al,n,j}$
after applying the function $g$.
In this example, $\al=2+i$ and $n=6$.}
\label{fig:changes_of_variables}
\end{figure}

In order to describe the asymptotic behavior of the extreme eigenvalues
$\la_{\al,n,1}$ and $\la_{\al,n,n}$,
we introduce the following notation:
\begin{equation}
\label{eq:notation_limit_extreme_eigenvalue}
s_\al
\eqdef
|\al|-2+\frac{1}{|\al|},\quad\text{i.e.}\quad
s_\al
=\frac{(|\al|-1)^2}{|\al|}
=\left(\sqrt{|\al|}-\frac{1}{\sqrt{|\al|}}\right)^2.
\end{equation}

\begin{theorem}[asymptotic expansion of the eigenvalues for strong perturbations]
\label{thm:strong_asympt}
Let $\al\in\bC$, $|\al|>1$.
As $n$ tends to infinity,
the extreme eigenvalues
of $A_{\al,n}$ converge exponentially
to $-s_\al$ and $4+s_\al$,
respectively:
\begin{align}
\label{eq:asympt_strong_left}
|\la_{\al,n,1}+s_\al|
&\le\frac{C_2(\al)}{|\al|^n},
\\[0.5ex]
\label{eq:asympt_strong_right}
|\la_{\al,n,n}-4-s_\al|
&\le\frac{C_2(\al)}{|\al|^n}.
\end{align}
Here $C_2(\al)$ is a positive constant depending only on $\al$.
For $2\le j\le n-1$,
the eigenvalues $\la_{\al,n,j}$
satisfy the asymptotic formulas~\eqref{eq:weak_lambda_asympt}.
\end{theorem}

According to Theorem~\ref{thm:strong_asympt},
for $|\al|>1$ we define $\laasympt_{\al,n,j}$ by~\eqref{eq:laasympt} when $2\le j\le n-1$, and for $j=1$ or $j=n$ we put
\begin{equation}\label{eq:laasympt_extreme}
\laasympt_{\al,n,1}\eqdef -s_\al,\qquad
\laasympt_{\al,n,j}\eqdef 4+s_\al.
\end{equation}
Formulas~\eqref{eq:laasympt} and~\eqref{eq:laasympt_extreme}
yield very efficient approximations of the eigenvalues, when $n$ is large enough.

Theorem~\ref{thm:strong_asympt} implies that for a fixed $\al$ with $|\al|>1$
and $n\to\infty$,
the number $s_\al$ is the
``asymptotical lower spectral gap''
and also the
``asymptotical upper spectral gap''
of the matrices $A_{\al,n}$,
in the following sense:
\begin{align*}
\lim_{n\to\infty}(\la_{\al,n,2}-\la_{\al,n,1})
&=0-(-s_\al)=s_\al,
\\
\lim_{n\to\infty}(\la_{\al,n,n}-\la_{\al,n,n-1})
&=(4+s_\al)-4=s_\al.
\end{align*}
Recall that for a sequence of Toeplitz matrices,
generated by a bounded real-valued symbol
\emph{not depending on $n$},
the eigenvalues asymptotically fill the whole interval between the essential infimum and the essential supremum of the symbol~\cite{Widom1994}.
Nevertheless, a splitting phenomenon is known for
the singular values of some sequences
of non-Hermitean Toeplitz matrices~\cite[Section~4.3]{BS1999}.

\bigskip

\section{Proofs for the case of weak perturbations}
\label{sec:proofs_weak}


Denote by $U_n$ the Chebyshev polynomial of the second kind of degree $n$.
It is well known that
\begin{equation}\label{eq:Chebyshev_trig_hyp}
U_n(\cos(x))=\frac{\sin((n+1)x}{\sin(x)},\quad
U_n(\cosh(x))=\frac{\sinh((n+1)x)}{\sinh(x)}.
\end{equation}

\begin{proposition}
\label{prop:char_pol_via_Cheb}
For every $\al,\la$ in $\bC$
and $n\ge 3$,
\begin{equation}\label{eq:char_pol_via_Cheb}
\det(\la I_n - A_{\al,n})
= U_n\left(\frac{\la-2}{2}\right)
-|\al|^2 U_{n-2}\left(\frac{\la-2}{2}\right)
-2(-1)^n\Re(\al).
\end{equation}
\end{proposition}

\begin{proof}
It is well known and easy to prove (using expansion by cofactors and induction) that $U_n(x)$ is the determinant of the tridiagonal Toeplitz matrix of the order $n$
with the entries $1,2x,1$.
In our notation, this means that
\begin{equation}\label{eq:char_pol_via_Cheb_alpha_zero}
\det(\la I_n-A_{0,n})
=U_n\left(\frac{\la-2}{2}\right).
\end{equation}
In order to prove~\eqref{eq:char_pol_via_Cheb},
we use expansion by cofactors and~\eqref{eq:char_pol_via_Cheb_alpha_zero}.
\end{proof}

\begin{proposition}
\label{prop:0_and_4_are_not_eigenvalues_weak}
Let $|\al|<1$ and $n\geq 3$.
Then $0$ and $4$ are not eigenvalues of $A_{\al,n}$.
\end{proposition}

\begin{proof}
Apply~\eqref{eq:char_pol_via_Cheb} with $\la=0$ and $\la=4$:
\begin{align*}
\det(-A_{\al,n})
&=(-1)^n\left(n(1-|\al|^2)+|1-\al|^2\right),
\\
\det(4 I_n - A_{\al,n})
&= n(1-|\al|^2)+|1-(-1)^n\al|^2.
\end{align*}
These expressions are obviously nonzero.
\end{proof}

If $\la\in(0,4)$, then we use the trigonometric change of variables $\la=g(x)$ in the characteristic polynomial.

\begin{proposition}
\label{prop:char_pol_trig}
Let $\al\in\bC$, $\al\notin\{-1,1\}$, $n\ge 3$, $x\in(0,\pi)$. Then
\begin{equation}
\label{eq:char_pol_trig_kl}
\det(g(x)I_n-A_{\al,n})
=\frac{(-1)^{n+1} |\al+1|^2
\left(\tan^2\frac{nx}{2}
-2k_\al\cot(x)\tan\frac{nx}{2}
-l_\al^2\right)}%
{1+\tan^2\frac{nx}{2}}.
\end{equation}
\end{proposition}

\begin{proof}
We start with~\eqref{eq:char_pol_via_Cheb},
write $\la$ as $2-2\cos(x)$,
use the parity or imparity of $U_n$
and the trigonometric formula~\eqref{eq:Chebyshev_trig_hyp}
for $U_n$:
\[
U_n(-\cos(x))
=(-1)^n U_n(\cos(x))
=(-1)^n\frac{\sin((n+1)x)}{\sin(x)}.
\]
Then
\begin{equation}\label{eq:char_pol_trig0}
\det(g(x)I_n-A_{\al,n})
=\frac{(-1)^n}{\sin(x)}
\left(\sin((n+1)x)
-|\al|^2\sin((n-1)x)
-2\Re(\al)\right).
\end{equation}
Applying the trigonometric identities
\[
\sin((n\pm 1)x)
=\sin(nx)\cos(x)\pm\cos(nx)\sin(x),
\]
\[
\sin(nx)
=\frac{2\tan\frac{nx}{2}}%
{1+\tan^2\frac{nx}{2}},\qquad
\cos(nx)
=\frac{1-\tan^2\frac{nx}{2}}%
{1+\tan^2\frac{nx}{2}},
\]
and regrouping the summands, we get
\begin{equation}\label{eq:char_pol_trig}
\begin{aligned}
\det(g(x) I_n&-A_{\al,n})
= \frac{(-1)^{n+1}}{1+\tan^2\frac{nx}{2}}\times
\\
&\times
\left(|\al+1|^2\tan^2{\frac{nx}{2}
-2(1-|\al|^2)\cot(x)\tan\frac{nx}{2}
-|\al-1|^2}\right),
\end{aligned}
\end{equation}
which is equivalent to~\eqref{eq:char_pol_trig_kl}.
\end{proof}

Notice that~\eqref{eq:char_pol_trig0},
up to a nonzero factor,
is a particular case of the expression
that appears in~\cite[eq. (3.10)]{YuehCheng2008}.

\begin{proposition}\label{prop:points_of_the_uniform_mesh_are_not_eigenvalues}
Let $\al\in\bC$, $\al\notin\{-1,1\}$, $n\ge 3$.
Then the points $g(j\pi/n)$, $1\le j\le n-1$,
are not eigenvalues of $A_{\al,n}$.
\end{proposition}

\begin{proof}
By~\eqref{eq:char_pol_trig_kl},
$\det(g(x)I_n-A_{\al,n})\ne0$
for $x$ of the form $j\pi/n$.
\end{proof}

For each $n\ge 3$ and $1\le j\le n$,
denote by $I_{n,j}$ the open interval $\left(\frac{(j-1)\pi}{n},\frac{j\pi}{n}\right)$.

For $|\al|<1$, $n\ge 3$, and $j$ in $\bZ$,
define $u_{\al,j}\colon(0,\pi)\to\bR$ by
\begin{equation}\label{eq:u_def}
u_{\al,j}(x)
\eqdef k_\al\cot(x)
+(-1)^{j+1}\sqrt{k_\al^2\cot^2(x)+l_\al^2}.
\end{equation}

\begin{proposition}\label{prop:weak_equation_tan}
Let $|\al|<1$, $n\ge3$, and $1\le j\le n$.
Then the equation
\begin{equation}\label{eq:weak_equation_tan}
\tan\frac{nx}{2} = u_{\al,j}(x)
\end{equation}
has a unique solution in~$I_{n,j}$,
and the corresponding value $g(x)$ is an eigenvalue of $A_{\al,n}$.
\end{proposition}

\begin{proof}
For $x$ in $I_{n,j}$,
the expression $\frac{(-1)^{n+1}|\al+1|^2}{1+\tan^2\frac{nx}{2}}$ takes finite nonzero values.
Omitting this factor, we consider the right-hand side of~\eqref{eq:char_pol_trig_kl} as a quadratic polynomial in $\tan\frac{nx}{2}$, with coefficients depending on $\al$ and $x$.
The roots of this quadratic polynomial are $u_{\al,1}$ and $u_{\al,2}$.
So, for $x$ in $I_{n,j}$,
the characteristic equation $\det(g(x)I_n-A_{\al,n})=0$
is equivalent to the union of the equations
\[
\tan\frac{nx}{2}=u_{\al,1}(x),\qquad
\tan\frac{nx}{2}=u_{\al,2}(x).
\]
Since $k_\al>0$, we get $u_{\al,1}(x)>0$ and $u_{\al,2}(x)<0$.
Furthermore, the first derivative of $u_{\al,1}$ and $u_{\al,2}$ is negative: 
\[
u_{\al,j}'(x)
=-\frac{k_\al}{\sin^2(x)}
\left(1+\frac{(-1)^{j+1}k_\al\cot{(x)}}{\sqrt{k_\al^2\cot^2{(x)} +l_\al^2}}\right)
<0,
\]
and the functions
$u_{\al,1}$ and $u_{\al,2}$
are strictly decreasing on $(0,\pi)$.
Their limit values are
\[
u_{\al,1}(0^+)=+\infty,\quad
u_{\al,1}(\pi^-)=0,\qquad
u_{\al,2}(0^+)=0,\quad
u_{\al,2}(\pi^-)=-\infty.
\]
For every $j$, the function
$x\mapsto\tan\frac{nx}{2}-u_{\al,j}(x)$ strictly increases on $I_{n,j}$
and changes its sign.
By the intermediate value theorem,
it has a unique zero on $I_{n,j}$.
Figure~\ref{fig:tg_u_weak} illustrates the ideas of this proof.
\end{proof}

\begin{figure}
\centering
\includegraphics{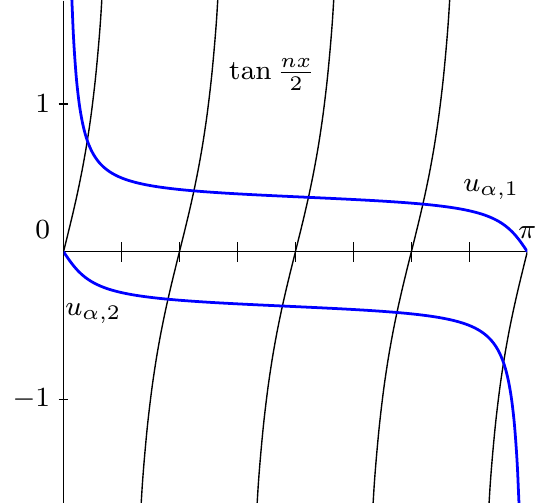}
\caption{Left-hand sides (black)
and right-hand sides (blue)
of equations~\eqref{eq:weak_equation_tan} for $\al=0.7+0.6i$, $n=8$, $1\le j\le n$.
Notice that $[0,\pi]$ is divided into $8$ equal subintervals, and each subinterval corresponds to its proper equation.
\label{fig:tg_u_weak}
}
\end{figure}

Proposition~\ref{prop:weak_equation_tan} implies Theorems~\ref{thm:weak_localization} and~\ref{thm:weak_equation} for $|\al|<1$.

Recall that $\eta$ is defined by~\eqref{eq:eta}.
A straightforward computation yields
\begin{equation}\label{eq:eta_der} 
\eta_{\al,j}'(x)
= -\frac{2k_\al
\left(1+\frac{(-1)^{j+1}k_\al\cot(x)}{\sqrt{k_\al^2\cot^2(x)+l_\al^2}}\right)
\left(1+\cot^2(x)\right)}%
{1+\left((-1)^{j+1}k_\al\cot(x)+\sqrt{k_\al^2\cot^2(x)+l_\al^2}\right)^2}.
\end{equation}

\begin{proposition}\label{prop:derivatives_eta}
Let $\al\in\bC$, $|\al|\ne 1$, $j\in\bZ$.
Then each derivative of $\eta_{\al,j}$ is a bounded function on $(0,\pi)$.
In particular, for every $x$ in $(0,\pi)$,
\begin{equation}\label{eq:der_eta_bound}
|\eta_{\al,j}'(x)|
\leq \frac{4(|\al|+1)}{||\al|-1|}.
\end{equation}
\end{proposition}

\begin{proof}
By~\eqref{eq:eta_der},
$\eta_{\al,j}'$ is analytic in a neighborhood of $x$,
for any $x$ in $(0,\pi)$.
Moreover, for $0<x<\pi/2$,
\[
\eta_{\al,j}'(x)
= -\frac{2k_\al
\left(1+\frac{(-1)^{j+1}k_\al}{\sqrt{k_\al^2+l_\al^2\tan^2(x)}}\right)
\left(1+\tan^2(x)\right)}%
{\tan^2(x)+\left((-1)^{j+1}k_\al+\sqrt{k_\al^2+l_\al^2\tan^2(x)}\right)^2},
\]
and for $\pi/2<x<\pi$, $\eta_{\al,j}'(x)$ has a similar expression, with $(-1)^j$ instead of $(-1)^{j+1}$.
Hence, $\eta_{\al,j}'$ has an analytic extension to a certain open set in the complex plane containing the segment $[0,\pi]$. Therefore, this function and all their derivatives are bounded on $(0,\pi)$.

An explicit upper bound for $|\eta_{\al,j}'|$
follows directly from~\eqref{eq:eta_der}.
Let us denote $\cot(x)$ by $t$ and explain briefly
how to ``supress'' the unbounded factor $1+t^2$
appearing in the numerator of~\eqref{eq:eta_der}.
If $(-1)^{k+1}k_\al t\ge 0$, then the denominator of~\eqref{eq:eta_der}
is sufficiently large:
\[
1+
\left((-1)^{j+1}k_\al t
+\sqrt{k_\al^2 t^2+l_\al^2}\right)^2
\ge k_\al^2 t^2+l_\al^2
\ge k_\al^2 t^2+k_\al^2
= k_\al^2 (t^2+1).
\]
If $(-1)^k k_\al t\ge 0$,
then a factor in the numerator of~\eqref{eq:eta_der}
is sufficiently small:
\begin{align*}
&1+\frac{(-1)^{j+1}k_\al t}{\sqrt{k_\al^2 t^2+1}}
=\frac{l_\al^2}{\left(\sqrt{k_\al^2 t^2+l_\al^2}+(-1)^j k_\al t\right)
\sqrt{k_\al^2 t^2+l_\al^2}}
\le \frac{l_\al^2}{k_\al^2 (t^2+1)}.
\end{align*}
In both cases, we easily obtain~\eqref{eq:der_eta_bound}.
\end{proof}

Let $f_{\al,n,j}$ be the function
defined on $[0,\pi]$ by the right-hand side of~\eqref{eq:weak_equation}:
\begin{equation}\label{eq:f}
f_{\al,n,j}(x)
\eqdef\frac{j\pi+\eta_{\al,j}(x)}{n}.
\end{equation}

\begin{proposition}\label{prop:weak_fixed_point}
Let $|\al|<1$, $N_1(\al)$ be defined by~\eqref{eq:N1}, $n>N_1(\al)$, and $1\le j\le n$.
Then $f_{\al,n,j}$ is contractive in $[0,\pi]$, and its fixed point belongs to $I_{n,j}$.
\end{proposition}

\begin{proof}
Since the function $\eta_{\al,j}$ takes values in $[-\pi,0]$ and its derivative is bounded by~\eqref{eq:der_eta_bound},
it is easy to see that
$f_{\al,n,j}(x)\in[0,\pi]$
for every $x$ in $[0,\pi]$,
and
\[
|f'_{\al,n,j}(x)|\le\frac{N_1(\al)}{n}<1.
\]
Moreover, the assumption $|\al|<1$ implies that
$\eta_{\al,j}(0)=0$
and $\eta_{\al,j}(\pi)=-\pi$.
Therefore
\[
f_{\al,n,j}(0)>0,\qquad f_{\al,n,j}(\pi)<\pi.
\]
So, if $x$ is the fixed point of $f_{\al,n,j}$, then $0<x<\pi$.
Thus, $-\pi<\eta_{\al,j}(x)<0$ and
\[
x=f_{\al,n,j}(x)
=\frac{j\pi+\eta_{\al,j}(x)}{n}
\in\left(\frac{(j-1)\pi}{n},\frac{j\pi}{n}\right)
=I_{n,j}.
\]
In particular, this implies that the fixed point of $f_{\al,n,j}$ coincides with $\tht_{\al,n,j}$.
\end{proof}

\begin{proposition}
\label{prop:weak_theta_asympt}
Let $\al\in\bC$, $|\al|<1$.
Then there exists $C_3(\al)>0$ such that
for $n$ large enough and $1\le j\le n$,
\begin{equation}\label{eq:weak_theta_asympt}
\tht_{\al,n,j}
=\frac{j\pi}{n}
+\frac{\eta_{\al,j}\left(\frac{j\pi}{n}\right)}{n}
+\frac{\eta_{\al,j}\left(\frac{j\pi}{n}\right)\eta_{\al,j}'\left(\frac{j\pi}{n}\right)}{n^2}
+r_{\al,n,j},
\end{equation}
where $|r_{\al,n,j}|\le\frac{C_3(\al)}{n^3}$.
\end{proposition}

\begin{proof}
Theorem~\ref{thm:weak_localization} assures the initial approximation
$\tht_{\al,n,j}=j\pi/n+O(1/n)$.
Substitute it into the right-hand side
of~\eqref{eq:weak_equation} and expand $\eta_{\al,j}$ by Taylor's formula around $j\pi/n$:
\[
\tht_{\al,n,j}
=\frac{j\pi+\eta_{\al,j}\left(\frac{j\pi}{n}+O\left(\frac{1}{n}\right)\right)}{n}
=\frac{j\pi}{n}+\frac{\eta_{\al,j}\left(\frac{j\pi}{n}\right)}{n}
+O_\al\left(\frac{1}{n^2}\right).
\]
Iterate once again in~\eqref{eq:weak_equation}:
\[
\tht_{\al,n,j}
=\frac{j\pi+\eta_{\al,j}\left(\frac{j\pi}{n}+\frac{\eta_{\al,j}\left(\frac{j\pi}{n}\right)}{n}
+O_\al\left(\frac{1}{n^2}\right)\right)}{n}.
\]
Expanding $\eta_{\al,j}$ around $j\pi/n$
with two exact term
and estimating the residue term with Proposition~\ref{prop:derivatives_eta}
we obtain the desired result.
\end{proof}

Theorem~\ref{thm:weak_asympt} follows from Proposition~\ref{prop:weak_theta_asympt}:
we just evaluate $g$ at the expression~\eqref{eq:weak_theta_asympt}
and expand it by Taylor's formula around $j\pi/n$.
See~\cite{BBGM2018} for a more general scheme.

The next proposition can be seen as a particular case of~\cite{YuehCheng2008}, therefore we do not include the proof.

\begin{proposition}[the eigenvectors for weak perturbations]
\label{prop:weak_eigenvectors}
Let $\al\in\bC$, $|\al|<1$,
$n\geq3$, $1\le j\le n$.
Then the vector
$v_{\al,n,j}=\bigl[v_{\al,n,j,k}\bigr]_{k=1}^n$
with components
\begin{equation}\label{eq:weak_eigvec_components}
v_{\al,n,j,k}
\eqdef \sin(k\tht_{\al,n,j})
+ \overline{\al}\sin((n-k)\tht_{\al,n,j})
\end{equation}
is an eigenvector of the matrix $A_{\al,n}$ associated to the eigenvalue $\la_{\al,n,j}$.
\end{proposition}

\section{\texorpdfstring{Case $\boldsymbol{|\al|=1}$}{Case |alpha|=1}}
\label{sec:abs_alpha_eq_one}

For $|\al|=1$, the eigenvalues of $A_{\al,n}$ can be computed explicitly.

\begin{proposition}\label{prop:case_abs_alpha_eq_1}
Let $\al\in\bC$, $|\al|=1$, $\al\ne\pm1$, $n\ge 3$, and $1\le j\le n$.
Then $\la_{\al,n,j}=g(\tht_{\al,n,j})$, where
\begin{equation}
\label{eq:tht_alpha_abs_1}
\tht_{\al,n,j}
= \frac{j\pi}{n} -\frac{2}{n}\arctan\left(l_\al^{(-1)^j}\right).
\end{equation}
Furthermore, $\tht_{\al,n,j}\in I_{n,j}$,
and the vector with components~\eqref{eq:weak_eigvec_components} is an eigenvector of $A_{\al,n}$ associated to $\la_{\al,n,j}$.
\end{proposition}
 
\begin{proof}
The condition about $\al$ implies that $k_\al=0$. In this case, the functions $\eta_{\al,1}$ and $\eta_{\al,2}$ are just constants:
\[
\eta_{\al,j}(x)
=-2\arctan\left(l_\al^{(-1)^j}\right),
\]
and the characteristic equation~\eqref{eq:weak_equation}
simplifies to the direct formula~\eqref{eq:tht_alpha_abs_1}.
\end{proof}

Proposition~\ref{prop:case_abs_alpha_eq_1} implies Theorems~\ref{thm:weak_localization}, \ref{thm:weak_equation}, and~\ref{thm:weak_asympt}, in the case $|\al|=1$ and $\al\ne\pm 1$.

For $\al=\pm 1$, the situation is different: the definition of $\eta_{\al,j}$ does not make sense, each number $\tht_{\al,n,j}$ coincides with one of the extremes of $I_{n,j}$,
and most of the eigenvalues are double.

\begin{proposition}\label{prop:case_alpha_pos_1}
Let $\al=1$, $n\ge 3$, and $1\le j\le n$.
Then $\la_{1,n,j}=g(\tht_{1,n,j})$, where
\begin{equation}\label{eq:tht_sol_alpha_1}
\tht_{1,n,j}
= \left(j-\frac{1-(-1)^j}{2}\right)\frac{\pi}{n}
=
\begin{cases}
\displaystyle\medstrut
\frac{2q\pi}{n}, & j=2q,
\\[1ex]
\displaystyle\medstrut
\frac{2q\pi}{n}, & j=2q+1.
\end{cases}
\end{equation}
The vector $v_{1,n,j}=[v_{1,n,j,k}]_{k=1}^n$
with components
\begin{equation}\label{eq:eigenvector_alpha_1}
v_{1,n,j,k}\eqdef
\sin\left(k\tht_{1,n,j}+\frac{(1-(-1)^j)\pi}{4}\right)
=
\begin{cases}
\displaystyle\medstrut
\sin\frac{2kq\pi}{n}, & j=2q,
\\[1ex]
\displaystyle\medstrut
\cos\frac{2kq\pi}{n}, & j=2q+1,
\end{cases}
\end{equation}
is an eigenvector of $A_{1,n}$ associated to $\la_{1,n,j}$.
\end{proposition}

\begin{proof}
The numbers $\tht_{1,n,j}$
can be found by passing to the limit $\al\to 1^-$
in~\eqref{eq:weak_equation}.
The equalities
$A_{1,n}v_{1,n,j}=\la_{1,n,j}v_{1,n,j}$
are easy to verify directly.
\end{proof}

The formulas for the eigenvalues of $A_{1,n}$ also follow from the theory of circulant matrices, since the matrix $A_{1,n}$ is circulant.
Notice that 
\[
\la_{1,n,1}
<\la_{1,n,2}=\la_{1,n,3}
<\la_{1,n,4}=\la_{1,n,5}
<\ldots,
\]
i.e. each of the eigenvalues $\la_{1,n,2}$, $\la_{1,n,4}$, etc. is double and has two linearly independent eigenvectors.

\begin{proposition}\label{prop:case_alpha_neg_1}
Let $\al=-1$, $n\ge 3$, and $1\le j\le n$.
Then $\la_{-1,n,j}=g(\tht_{-1,n,j})$, where
\begin{equation}\label{eq:tht_sol_alpha_minus_1}
\tht_{-1,n,j}
= \left(j-\frac{1+(-1)^j}{2}\right)\frac{\pi}{n}
=
\begin{cases}
\displaystyle\medstrut
\frac{(2q-1)\pi}{n}, & j=2q-1,
\\[1ex]
\displaystyle\medstrut
\frac{(2q-1)\pi}{n}, & j=2q.
\end{cases}
\end{equation}
The vector $v_{-1,n,j}=[v_{-1,n,j,k}]_{k=1}^n$
with components
\begin{equation}\label{eq:eigenvector_alpha_minus_1}
v_{-1,n,j,k}\eqdef
\sin\left(k\tht_{-1,n,j}+\frac{(1+(-1)^j)\pi}{4}\right)
=
\begin{cases}
\displaystyle\medstrut
\sin\frac{k(2q-1)\pi}{n}, & j=2q-1,
\\[1ex]
\displaystyle\medstrut
\cos\frac{k(2q-1)\pi}{n}, & j=2q,
\end{cases}
\end{equation}
is an eigenvector of $A_{-1,n}$ associated to $\la_{-1,n,j}$.
\end{proposition}

\begin{proof}
Similar to the proof of Proposition~\ref{prop:case_alpha_pos_1}.
\end{proof}

By Proposition~\ref{prop:case_alpha_neg_1},
\[
\la_{-1,n,1}=\la_{-1,n,2}
<\la_{-1,n,3}=\la_{-1,n,4}
<\ldots,
\]
i.e. each of the eigenvalues
$\la_{-1,n,1}$, $\la_{-1,n,3}$, etc.
is double and has two linearly independent eigenvectors.

\section{Proofs for the case of strong perturbations}
\label{sec:proofs_strong}

Let us show that for $|\al|>1$, $n$ large enough and $2\le j\le n-1$, the situation is nearly the same as in Proposition~\ref{prop:weak_fixed_point}.
Recall that $N_1(\al)$ and $f_{\al,n,j}$ are defined by~\eqref{eq:N1} and~\eqref{eq:f}.

\begin{proposition}\label{prop:strong_equation_tan_and_fixed_point}
Let $|\al|>1$, $n>N_1(\al)$,
and $2\le j\le n-1$.
Then $f_{\al,n,j}$ is contractive on $[0,\pi]$.
If $x$ is the fixed point of $f_{\al,n,j}$,
then $x\in I_{n,j}$, and $g(x)$ is an eigenvalue of $A_{\al,n}$.
\end{proposition}

\begin{proof}
Inequality~\eqref{eq:der_eta_bound} and the assumption $n\ge N_1(\al)$ imply that $f_{\al,n,j}$ is contractive on $[0,\pi]$.
Unlike in the case of weak perturbations,
in the case $|\al|>1$ we have
\begin{equation}\label{eq:eta_extreme_values_for_alpha_greater_1}
\eta_{\al,j}(0)=-\pi,\qquad
\eta_{\al,j}(\pi)=0.
\end{equation}
Now the condition $2\le j\le n-1$
assures that $0$ and $\pi$ are not fixed points of $f_{\al,n,j}$:
\[
f_{\al,n,j}(0)=\frac{j\pi-\pi}{n}>0,\qquad
f_{\al,n,j}(\pi)=\frac{j\pi+0}{n}<\pi.
\]
Let $x$ be the fixed point of $f_{\al,n,j}$.
Then $0<x<\pi$.
Hence $-\pi<\eta_{\al,j}(x)<0$
and $x=f_{\al,n,j}(x)\in I_{n,j}$.
\end{proof}

\begin{proposition}
\label{prop:0_and_4_are_not_eigenvalues_strong}
Let $|\al|>1$ and $n>N_1(\al)$.
Then $0$ and $4$ are not eigenvalues of $A_{\al,n}$.
\end{proposition}

\begin{proof}
Formulas from the proof of Proposition~\ref{prop:0_and_4_are_not_eigenvalues_weak} and the assumption $n>N_1(\al)$ easily imply that
$\det(\la I_n-A_{\al,n})\ne 0$
for $\la\in\{0,4\}$.
\end{proof}

If $|\al|>1$ and $n\ge N_1(\al)$,
then the functions $f_{\al,n,1}$ and $f_{\al,n,n}$ are contractive,
but their fixed points are $0$ and $\pi$.
The corresponding values of $g$, i.e. the points $0$ and $4$, are not eigenvalues of $A_{\al,n}$.
Hence, for $|\al|>1$ and $n$ large enough,
the trigonometric change of variables $\la=g(x)$ with real $x$
allows us to find only $n-2$ eigenvalues.
In order to find the eigenvalues outside of $[0,4]$, we will use the change of variables $\la=g_-(x)$ or $\la=g_+(x)$,
defined by~\eqref{eq:g_plus_minus}.

\begin{proposition}\label{prop:chareq_strong_hyper}
For $x>0$ and $\la=g_-(x)$,
the equation $\det(\la I_n-A_{\al,n})=0$ is equivalent to
\begin{equation}\label{eq:chareq_tanh_first}
\tanh(x)
=\frac{2(|\al|^2-1)\tanh\frac{nx}{2}}%
{|\al+1|^2\tanh^2\frac{nx}{2}
+|\al-1|^2}.
\end{equation}
For $x>0$ and $\la=g_+(x)$,
the equation $\det(\la I_n-A_{\al,n})=0$ is equivalent to
\begin{equation}\label{eq:chareq_tanh_last}
\tanh(x)
=\frac{2(|\al|^2-1)\tanh\frac{nx}{2}}%
{|\al+(-1)^n|^2\tanh^2\frac{nx}{2}
+|\al-(-1)^n|^2}.
\end{equation}
\end{proposition}

\begin{proof}
For $\la=g_-(x)=2-2\cosh(x)$,
we apply~\eqref{eq:char_pol_via_Cheb}
and~\eqref{eq:Chebyshev_trig_hyp}.
After some simple transformations,
\begin{equation}\label{eq:charpol_strong_first}
\begin{aligned}
\det(g_-(x)I_n-A_{\al,n})
&=\frac{(-1)^n}{1-\tanh^2\frac{nx}{2}}
\biggl(|\al+1|^2\tanh^2\frac{nx}{2}
\\
&\quad-2(|\al|^2-1)\tanh\frac{nx}{2}\,\coth(x)
+|\al-1|^2\biggr).
\end{aligned}
\end{equation}
This expression for the characteristic polynomial yields~\eqref{eq:chareq_tanh_first}.
The proof of~\eqref{eq:chareq_tanh_last} is analogous.
\end{proof}

\begin{remark}\label{rem:rapid_and_slow}
In formula~\eqref{eq:char_pol_trig},
$\tan\frac{nx}{2}$ is
rapidly oscillating
and $\cot(x)$ is much slower,
therefore we solve~\eqref{eq:char_pol_trig} for $\tan\frac{nx}{2}$.
The situation in~\eqref{eq:charpol_strong_first}
is different:
if $x$ is separated from zero and $n$ is large enough,
then $\tanh\frac{nx}{2}$ is almost a constant,
and we prefer to solve~\eqref{eq:charpol_strong_first} for $\tanh(x)$.
\end{remark}

In what follows, we restrict ourselves to the analysis of the equation~\eqref{eq:chareq_tanh_first},
because~\eqref{eq:chareq_tanh_last} is similar.
Define $\psi_\al\colon[0,1]\to[0,+\infty)$ by
\begin{equation}\label{eq:psi}
\psi_\al(t)\eqdef\frac{2(|\al|^2-1)t}%
{|\al+1|^2 t^2+|\al-1|^2}.
\end{equation}
Notice that
\begin{equation}\label{eq:psi_al_1}
\psi_\al(1)
=\frac{2(|\al|^2-1)}{|\al+1|^2+|\al-1|^2}
=\frac{|\al|^2-1}{|\al|^2+1}
=\tanh(\log|\al|).
\end{equation}
We are going to construct explicitly a left neighborhood of $1$ where the values of $\psi_\al$ are close enough to $\tanh(\log|\al|)$.

\begin{lemma}\label{lem:psi}
Let $|\al|>1$. Then for every $t$ with
\begin{equation}\label{eq:t_assumption}
1-\frac{|\al|-1}{(|\al|+1)^3} \le t\le 1,
\end{equation}
the following inequalities hold:
\begin{equation}\label{eq:Dpsi_bound}
|\psi_\al'(t)|\le 1,
\end{equation}
\begin{equation}\label{eq:psi_bound}
\tanh\frac{\log|\al|}{2}
\le\psi_\al(t)
\le\tanh\frac{3\log|\al|}{2},
\end{equation}
\begin{equation}\label{eq:one_minus_psi_square_bound}
1 - \psi_\al^2(t) \ge\frac{2}{(|\al|+1)^3}.
\end{equation}
\end{lemma}

\begin{proof}
The proof is quite elementary, therefore we will only mention the main steps.
Assumption~\eqref{eq:t_assumption} implies that
\[
1-t^2\le \frac{2(|\al|-1)}{(|\al|+1)^3},\qquad
|\al+1|^2 t^2+|\al-1|^2
\ge\frac{2(|\al|^3+|\al|^2+2)}{|\al|+1}.
\]
With these estimates we obtain~\eqref{eq:Dpsi_bound}.
After that, the mean value theorem and~\eqref{eq:psi_al_1} provide~\eqref{eq:psi_bound}.
Inequality~\eqref{eq:one_minus_psi_square_bound} follows from~\eqref{eq:psi_bound}.
\end{proof}

Recall that $N_2(\al)$ and $S_\al$
are given by~\eqref{eq:N2}
and~\eqref{eq:segment}.
Define $h_{\al,n}$ on $S_\al$
as the right-hand side of~\eqref{eq:eq_first}:
\[
h_{\al,n}(x)
\eqdef\arctanh\frac{2(|\al|^2-1)\tanh\frac{nx}{2}}%
{|\al+1|^2\tanh^2\frac{nx}{2}+|\al-1|^2}.
\]

\begin{proposition}\label{prop:h_contractive}
Let $\al\in\bC$, $|\al|>1$, and $n>N_2(\al)$.
Then $h_{\al,n}$ is contractive on $S_\al$,
and its fixed point is the solution of~\eqref{eq:chareq_tanh_first}.
\end{proposition}

\begin{proof}
We represent $h_{\al,n}$
as the following composition:
\[
h_{\al,n}(x)
=\arctanh\left(\psi_\al\left(\tanh\frac{nx}{2}\right)\right).
\]
For $x$ in $S_\al$,
denote $\tanh\frac{nx}{2}$ by $t$.
Then
\[
1-t\le 2e^{-nx}
\le 2e^{-N_2(\al)\frac{\log|\al|}{2}}
<\frac{|\al|-1}{(|\al|+1)^3}.
\]
Therefore, by~\eqref{eq:one_minus_psi_square_bound} we have $\psi_\al(\tanh\frac{nx}{2})<1$,
and the definition of $h_{\al,n}$ makes sense.
By~\eqref{eq:psi_bound},
$h_{\al,n}$ takes values in $S_\al$.
Estimate from above the derivative of $h_{\al,n}$ using~\eqref{eq:one_minus_psi_square_bound}, \eqref{eq:Dpsi_bound},
and the elementary inequality $u e^{-u}\le 1/e$:
\begin{align*}
|h_{\al,n}'(x)|
&\le\frac{|\psi_\al'(t)|}{1-\psi_\al^2(t)}\cdot
\frac{n}{2\cosh^2\frac{nx}{2}}
\le (|\al|+1)^3 ne^{-nx}
\\
&\le (|\al|+1)^3 ne^{-\frac{n\log|\al|}{2}}
=(|\al|+1)^3 n e^{-\frac{n\log|\al|}{4}}
e^{-\frac{n\log|\al|}{4}}
\\
&\le (|\al|+1)^3 \cdot \frac{4}{\log|\al|}
\cdot 
 \frac{\log|\al|}{(|\al|+1)^5}
=\frac{4}{(|\al|+1)^2}<1.
\end{align*}
Obviously, the fixed point of $h_{\al,n}$ is the solution of~\eqref{eq:eq_first} and \eqref{eq:chareq_tanh_first}.
\end{proof}

At the moment, we have proven~Theorems~\ref{thm:strong_localization} and \ref{thm:strong_equation}.
Asymptotic formulas for $\la_{\al,n,j}$ with $|\al|>1$ and $2\le j\le n-1$ can be justified in the same manner as for $|\al|<1$, and we are left to prove the exponential convergence~\eqref{eq:asympt_strong_left} and~\eqref{eq:asympt_strong_right}.

\begin{proposition}
\label{prop:eq:diff_thtleft_logal}
Let $|\al|>1$ and
$C_4(\al)\eqdef |\al|^3 e^{\frac{|\al|^3}{\log|\al|}}$.
Then for all $n\ge N_2(\al)$
\begin{equation}\label{eq:diff_thtleft_logal}
\bigl|\tht_{\al,n,1}-\log|\al|\bigr|
\le \frac{C_4(\al)}{|\al|^n}.
\end{equation}
\end{proposition}

\begin{proof}
For brevity, put $x=\tht_{\al,n,1}$.
Apply the mean value theorem to $\psi_\al$, taking into account~\eqref{eq:Dpsi_bound}:
\[
\bigl|\tanh(x)-\tanh(\log|\al|)\bigr|
=\left|\psi_\al\left(\tan\frac{nx}{2}\right)
-\psi_\al(1)\right|
\le 1-\tan\frac{nx}{2}
\le 2e^{-nx}.
\]
On the other hand,
apply the mean value theorem to $\tanh$:
\[
|\tanh(x)-\tanh(\log|\al|)|
\ge \frac{\bigl|x-\log|\al|\bigr|}{\cosh^2\frac{3\log|\al|}{2}}
\ge \frac{2}{|\al|^3} \bigl|x-\log|\al|\bigr|.
\]
From this chain of inequalities,
\begin{equation}\label{eq:x_minus_limit}
|x-\log|\al||\le |\al|^3 e^{-nx}.
\end{equation}
We already know from Proposition~\ref{prop:h_contractive} that~$x\ge\frac{\log|\al|}{2}$.
Thus,
\[
x \ge \log|\al| - |\al|^3 e^{-\frac{n\log|\al|}{2}}.
\]
Using the elementary inequality $u e^{-u}\le 1/e$ we get
\begin{equation}\label{eq:nx_ge}
nx
\ge 
n\log|\al| - |\al|^3\,ne^{-\frac{n\log|\al|}{2}}
\ge 
n\log|\al|-\frac{|\al|^3}{\log|\al|}.
\end{equation}
By~\eqref{eq:x_minus_limit} and~\eqref{eq:nx_ge},
inequality \eqref{eq:diff_thtleft_logal} holds.
\end{proof}

In a similar manner,
$|\tht_{\al,n,n}-\log|\al||\le C_4(\al)/|\al|^n$.
Since the derivatives of $g_-$ and $g_+$
are bounded on $\left[0,\frac{3}{2}\log|\al|\right]$,
we get limit relations~\eqref{eq:asympt_strong_left}
and~\eqref{eq:asympt_strong_right}.
Thereby we have proven the parts of Theorem~\ref{thm:strong_asympt} related to the extreme eigenvalues $\la_{\al,n,1}$ and $\la_{\al,n,n}$.

\begin{proposition}[the eigenvectors for strong perturbations]
\label{prop:strong_eigenvectors}
Let $\al\in\bC$, $|\al|>1$, $n\ge N(\al)$.
Then the vectors
$v_{\al,n,1}\eqdef[v_{\al,n,1,k}]_{k=1}^n$
and
$v_{\al,n,n}\eqdef[v_{\al,n,n,k}]_{k=1}^n$
with components
\begin{align}
\label{eq:strong_eigvec_components_1}
v_{\al,1,n,k}
&= \sinh(k\tht_{\al,n,1}) + \overline{\al}\sinh((n-k)\tht_{\al,n,1}),
\\
\label{eq:strong_eigvec_components_n}
v_{\al,n,n,k}
&= (-1)^k\sinh(k\tht_{\al,n,n}))
+ (-1)^{k+n}\,\overline{\al}
\sinh((n-k)\tht_{\al,n,n}),
\end{align}
are the eigenvectors of the matrix $A_{\al,n}$ associated to the eigenvalues $\la_{\al,1,n}$ and $\la_{\al,n,n}$, respectively.
For $2\le j\le n-1$,
the vector $v_{\al,n,j}$
defined by \eqref{eq:weak_eigvec_components}
is an eigenvector of $A_{\al,n}$
associated to the eigenvalue $\la_{\al,n,j}$.
\end{proposition}

\begin{remark}
It is possible to show that for a fixed $\al$ with $|\al|>1$ and large values of $n$, the norms of the vectors $v_{\al,n,1}$ and $v_{\al,n,n}$,
given by~\eqref{eq:strong_eigvec_components_1} and~\eqref{eq:strong_eigvec_components_n}, grow as $|\al|^n$.
In order to avoid large numbers,
we recommend to divide each component of these vectors by $|\al|^n$.
\end{remark}

\section{Numerical experiments}
\label{sec:numerical}

We use the following notation for different approximations of the eigenvalues.
\begin{itemize}
\item $\la_{\al,n,j}^{\text{gen}}$ are the eigenvalues computed in Sagemath by general algorithms, with double-precision arithmetic.
\item $\la_{\al,n,j}^{\text{fp}}$ are the eigenvalues computed by formulas of Theorems~\ref{thm:weak_equation} and~\ref{thm:strong_equation},
i.e. solving the equations~\eqref{eq:weak_equation},
~\eqref{eq:eq_first}, and \eqref{eq:eq_last}
by the fixed point iteration;
these computations are performed
in the high-precision  arithmetic
with $3322$ binary digits
($\approx 1000$ decimal digits).
Using~$\la_{\al,n,j}^{\text{fp}}$
we compute $v_{\al,n,j}$ by~\eqref{eq:weak_eigvec_components},
\eqref{eq:strong_eigvec_components_1},
and~\eqref{eq:strong_eigvec_components_n}.

\item $\la_{\al,n,j}^{\text{asympt}}$
are the approximations given by~\eqref{eq:laasympt} and~\eqref{eq:laasympt_extreme}.
\end{itemize}

In~\eqref{eq:eq_first} and~\eqref{eq:eq_last},
we compute $\tanh\frac{nx}{2}$
as $1-2e^{-nx}/(1+e^{-nx})$,
because $nx/2$ can be large
and the standard formula for $\tanh$ can produce overflows (``NaN'').

We have constructed a large series of examples with random values of $\al$ and $n$.
In all these examples, we have obtained
\[
\max_{1\le j\le n}|\la_{\al,n,j}^{\text{gen}}-\la_{\al,n,j}^{\text{fp}}|< 2\cdot 10^{-13},\qquad
\frac{\|A_{\al,n}v_{\al,n,j}
-\la_{\al,n,j}^{\text{fp}}v_{\al,n,j}\|}{\|v_{\al,n,j}\|}
<10^{-996}.
\]
This means that the exact formulas from
Theorems~\ref{thm:weak_equation} and~\ref{thm:strong_equation}
are fulfilled up to the rounding errors.
Theorems~\ref{thm:weak_localization}
and~\ref{thm:strong_localization}
can be viewed as simple corollaries from
Theorems~\ref{thm:weak_equation} and~\ref{thm:strong_equation},
so they do not need additional tests.
For Theorems~\ref{thm:weak_asympt} and~\ref{thm:strong_asympt},
we have computed the errors
\[
R_{\al,n,j}\eqdef \la_{\al,n,j}^{\text{asympt}}-\la_{\al,n,j}^{\text{fp}}
\]
and their maximums
$\|R_{\al,n}\|_\infty=\max_{1\le j\le n}|R_{\al,n,j}|$.
Tables~\ref{table:errors_weak} and~\ref{table:errors_strong}
show that these errors indeed can be bounded by $C_1(\al)/n^3$,
and $C_1(\al)$ has to take bigger values when $|\al|$ is close to $1$.

\begin{table}[ht]
\caption{Values of $\|R_{\al,n}\|_\infty$
and $n^3 \|R_{\al,n}\|_\infty$
for some $|\al|<1$.
\label{table:errors_weak}}
\[
\begin{array}{|c|c|c|}
\hline
\multicolumn{3}{ |c| }{\bigstrut\al=-0.3+0.5i,\quad |\al|\approx 0.58}
\\\hline
\bigstrut n & \|R_{\al,n}\|_\infty &%
\hstrut{}n^3 \|R_{\al,n}\|_\infty\hstrut{}%
\\\hline
\medstrut 64 & 1.76\times10^{-4} & 46.05 \\
\medstrut 128 & 2.49\times10^{-5} & 52.13 \\
\medstrut 256 & 3.29\times10^{-6} & 55.12 \\
\medstrut 512 & 4.22\times10^{-7} & 56.58 \\
\medstrut 1024 & 5.34\times10^{-8} & 57.31 \\
\medstrut 2048 & 6.71\times10^{-9} & 57.67 \\
\medstrut\hstrut{}4096\hstrut{}&%
\hstrut{}8.42\times10^{-10}\hstrut{}&%
\hstrut{}57.84\hstrut{}\\
\medstrut\hstrut{}8192\hstrut{}&%
\hstrut{}1.05\times10^{-10}\hstrut{}&%
\hstrut{}57.93\hstrut{}\\
\hline
\end{array}
\qquad
\begin{array}{|c|c|c|}
\hline
\multicolumn{3}{ |c| }{\bigstrut\al=0.7+0.6i,\quad |\al|\approx 0.92}
\\\hline
\bigstrut n & \|R_{\al,n}\|_\infty &%
\hstrut{}n^3 \|R_{\al,n}\|_\infty\hstrut{}%
\\\hline
\medstrut 64 & 1.02\times10^{-3} & 266.71 \\
\medstrut 128 & 1.59\times10^{-4} & 333.02 \\
\medstrut 256 & 2.24\times10^{-5} & 376.61 \\
\medstrut 512 & 2.99\times10^{-6} & 401.28 \\
\medstrut 1024 & 3.86\times10^{-7} & 414.29 \\
\medstrut 2048 & 4.90\times10^{-8} & 420.94 \\
\medstrut\hstrut{}4096\hstrut{}&%
\hstrut{}6.17\times10^{-9}\hstrut{}&%
\hstrut{}424.30\hstrut{}\\
\medstrut\hstrut{}8192\hstrut{}&%
\hstrut{}7.75\times 10^{-10}\hstrut{}&%
425.99 \\
\hline
\end{array}
\]
\end{table}

\begin{table}[ht]
\caption{Values of $\|R_{\al,n}\|_\infty$
and $n^3 \|R_{\al,n}\|_\infty$
for some $|\al|>1$.
\label{table:errors_strong}}
\[
\begin{array}{|c|c|c|}
\hline
\multicolumn{3}{|c|}{\bigstrut\al=2+i,\quad |\al|\approx 2.23}
\\\hline
\bigstrut n & \|R_{\al,n}\|_\infty &%
\hstrut{}n^3 \|R_{\al,n}\|_\infty\hstrut{}%
\\\hline
\medstrut 64 & 1.55\times 10^{-4} & 40.59 \\
\medstrut 128 & 2.15\times 10^{-5} & 45.18 \\
\medstrut 256 & 2.82\times 10^{-6} & 47.33 \\
\medstrut 512 & 3.60\times 10^{-7} & 48.36 \\
\medstrut 1024 & 4.55\times 10^{-8} & 48.86 \\
\medstrut 2048 & 5.72\times 10^{-9} & 49.10 \\
\medstrut\hstrut{}4096\hstrut{}&%
\hstrut{}7.16\times 10^{-10}\hstrut{}&%
\hstrut{}49.22\hstrut{}\\
\medstrut\hstrut{}8192\hstrut{}&%
\hstrut{}8.97\times 10^{-11}\hstrut{}&%
\hstrut{}49.29\hstrut{}%
\\
\hline
\end{array}
\qquad
\begin{array}{|c|c|c|}
\hline
\multicolumn{3}{ |c| }{\bigstrut\al=0.8-0.7i,\quad |\al|\approx 1.06}
\\\hline
\bigstrut n & \|R_{\al,n}\|_\infty &%
\hstrut{}n^3 \|R_{\al,n}\|_\infty\hstrut{}%
\\\hline
\medstrut 64 & 2.19\times10^{-4} & 57.51 \\
\medstrut 128 & 2.19\times10^{-5} & 45.90 \\
\medstrut 256 & 1.40\times10^{-5} & 235.36 \\
\medstrut 512 & 2.99\times10^{-6} & 401.90 \\
\medstrut 1024 & 4.55\times10^{-7} & 488.25 \\
\medstrut 2048 & 6.16\times10^{-8} & 528.84 \\
\medstrut\hstrut{}4096\hstrut{}&%
\hstrut{}7.98\times10^{-9}\hstrut{}&%
\hstrut{}548.04\hstrut{}\\
\medstrut\hstrut{}8192\hstrut{}&%
\hstrut{}1.01\times10^{-9}\hstrut{}&%
\hstrut{}557.32\hstrut{}\\
\hline
\end{array}
\]
\end{table}

We have also tested~\eqref{eq:asympt_strong_left} and~\eqref{eq:asympt_strong_right}.
As $n$ grows,
$|\al|^n |R_{\al,n,1}|$
and $|\al|^n |R_{\al,n,n}|$ approach rapidly the same limit value depending on $\al$. 
For example,
\begin{align*}
&\text{for}\ \al=2+i, &
&\lim_{n\to\infty}\left(|\al|^n |R_{\al,n,1}|\right)\approx 2.86,
\\
&\text{for}\ \al=0.8-0.7i, &
&\lim_{n\to\infty}\left(|\al|^n |R_{\al,n,1}|\right)\approx 1.12\cdot 10^{-2}.
\end{align*}

\paragraph{Acknowledgements.}
The research has been supported
by CONACYT (Mexico) scholarships
and by IPN-SIP project 20200650
(Instituto Polit\'{e}cnico Nacional, Mexico).
We are grateful to \'{O}scar Garc\'{i}a Hern\'{a}ndez for his participation on the early stages of this research
(we worked with the case $0\le\al<1$).

\bigskip\noindent
Sergei M. Grudsky,\\
\url{grudsky@math.cinvestav.mx},\\
ReseachID J-5263-2017,\\
CINVESTAV, Departamento de Matem\'aticas,\\
Apartado Postal 07360, Ciudad de M\'exico,\\
Mexico.

\bigskip\noindent
Egor A. Maximenko,\\
\url{emaximenko@ipn.mx},\\
\url{http://orcid.org/0000-0002-1497-4338},\\
Instituto Polit\'ecnico Nacional,
Escuela Superior de F\'isica y Matem\'aticas,\\
Apartado Postal 07730, Ciudad de M\'exico,\\
Mexico.

\bigskip\noindent
Alejandro Soto-Gonz\'{a}lez,\\
\url{asoto@math.cinvestav.mx},\\
\url{https://orcid.org/0000-0003-2419-4754},\\
CINVESTAV, Departamento de Matem\'aticas,\\
Apartado Postal 07360, Ciudad de M\'exico,\\
Mexico.

\end{document}